\documentclass[12pt, reqno]{amsart}

\usepackage{amsthm,amssymb,amstext,amscd,amsfonts,amsbsy,amsrefs,amsxtra,latexsym,amsmath,xcolor,mathrsfs,fancybox,upgreek, soul,url}
\usepackage[english]{babel}
\usepackage[all,cmtip]{xy}
\usepackage[latin1]{inputenc}
\usepackage{cancel}
\usepackage[draft]{hyperref}
\usepackage{comment}
\usepackage{mdframed}
\allowdisplaybreaks
\usepackage{mathtools}
\usepackage{enumerate}
\usepackage{thmtools}
\usepackage{thm-restate}
\usepackage{chngcntr}
\usepackage{enumitem}
\usepackage{etoolbox}

\DeclarePairedDelimiter\abs{\lvert}{\rvert}%
\DeclarePairedDelimiter\norm{\lVert}{\rVert}%

\makeatletter
\let\oldabs\abs
\def\abs{\@ifstar{\oldabs}{\oldabs*}}
\let\oldnorm\norm
\def\norm{\@ifstar{\oldnorm}{\oldnorm*}}
\makeatother

\newtheorem{theorem}{Theorem}
\newtheorem{lemma}[theorem]{Lemma}

\newtheorem{proposition}[theorem]{Proposition}

\theoremstyle{definition}

\theoremstyle{remark}

\newtheorem*{remark}{Remark}

\counterwithin{equation}{enumi}

\numberwithin{theorem}{section}
\numberwithin{proposition}{section}
\numberwithin{lemma}{section}
\numberwithin{corollary}{section}
\numberwithin{conjecture}{section}

\newcommand{\Z}{\mathbb{Z}}

\newcommand{\C}{\mathbb{C}}

\newcommand{\sech}{\textnormal{sech}}

\newcommand{\re}{\textnormal{Re}}

\newcommand\blfootnote[1]{%
	\begingroup
	\renewcommand\thefootnote{}\footnote{#1}%
	\addtocounter{footnote}{-1}%
	\endgroup
}

\def\H{\mathbb{H}}

\begin{document}

\title[The asymptotic profile of an eta-theta quotient]{The asymptotic profile of an eta-theta quotient related to entanglement entropy in string theory}

\author{Joshua Males}

\address{University of Cologne, Department of Mathematics and Computer Science, Weyertal 86-90, 50931
	Cologne, Germany}
\email{jmales@math.uni-koeln.de}

\begin{abstract}
	In this paper we investigate a certain eta-theta quotient which appears in the partition function of entanglement entropy. Employing Wright's circle method, we give its bivariate asymptotic profile.
\end{abstract}



\maketitle

\section{Introduction and Statement of Results}
Modern mathematical physics in the direction of string theory and black holes is intricately linked to number theory. For example, work of Dabholkar, Murthy, and Zagier relates certain mock modular forms to physical phenomena such as quantum black holes and wall crossing \cite{dabholkar2012quantum}. Similarly, the connections between automorphic forms and a second quantised string theory are described in \cite{dijkgraaf1997elliptic}, and modular forms for certain elliptic curves and their realisation in string theory is discussed in \cite{kondo2019string}. Further, the recent paper \cite{harvey2019ramanujans} discusses in-depth the links between work of the enigmatic Ramanujan in relation to modular forms and their generalisations and string theoretic objects (and indeed, why such links should be expected).

Knowledge of the behaviour of the modular objects aids the descriptions of physical phenomena. For instance, in \cite{gliozzi1977supersymmetry}, the authors use the classical number-theoretic Jacobi triple product identity to demonstrate the supersymmetry of the open-string spectrum using RNS fermions in light-cone gauge (see also \cite{witten2019open}).
In particular, parts of physical partition functions are often modular or mock modular objects. For example, the partition functions of the Melvin model \cite{RUSSO1996131} and the conical entropy of both the open and closed superstring \cite{he2015notes} both involve the weight $-3$ and index $0$ meromorphic Jacobi form 
\begin{equation*}
f(z;\tau) \coloneqq \frac{\vartheta(z;\tau)^4}{ \eta(\tau)^9 \vartheta(2z;\tau)},
\end{equation*}
where $\eta$ is the Dedekind eta function given by
\begin{equation*}
\eta(\tau) \coloneqq q^{\frac{1}{24}} \prod_{n \geq 1} \left( 1-q^n \right),
\end{equation*}
and
\begin{equation*}
\vartheta(z;\tau) \coloneqq i \zeta^{\frac{1}{2}} q^{\frac{1}{8}} \prod_{n\geq 1} (1-q^n)(1-\zeta q^n)(1-\zeta^{-1} q^{n-1})
\end{equation*}
is the Jacobi theta function, with $\zeta \coloneqq e^{2 \pi i z}$ for $z \in \C$, and $q \coloneqq e^{2 \pi i \tau}$ with $\tau \in \H$, the upper half-plane.

We are particularly interested in the coefficients of the $q$-expansion of $f$ where $0 \leq z \leq 1$, away from the pole at $z = 1/2$, where the residue of $f$ is calculated in \cite{witten2019open} - the other residues may be calculated using the elliptic transformation formulae for $f$. For instance, the asymptotic behaviour of the coefficients is required in order to investigate the UV limit. For a fixed value of $z$ the problem of finding the asymptotics of the coefficients is elementary, as \cite{he2015notes} notes. In particular, fixing $z = \frac{h}{k}$ a rational number with $\gcd(h,k) = 1$ and $0 \leq h < \frac{k}{2}$, then classical results in the theory of modular forms (see Theorem 15.10 of \cite{bringmann2017harmonic} for example) give that the coefficients of $f(\frac{h}{k};\tau) = \sum_{n \geq 0} a_{h,k}(n)q^n$ behave asymptotically as
\begin{equation*}
a_{h,k}(n) \sim \frac{\left( \frac{h}{k} \right)^{\frac{7}{4}}}{2 \sqrt{2} \pi} n^{-\frac{9}{4}} e^{4 \pi \sqrt{ \frac{hn}{k}}}.
\end{equation*}

In the present paper, we let
\begin{equation*}
f(z;\tau) \eqqcolon \sum_{\substack{n \geq 0 \\ m \in \Z}} b(m,n) \zeta^m q^n.
\end{equation*}
and investigate the coefficients $b(m,n)$; in particular we want to compute the bivariate asymptotic profile of $b(m,n)$ for a certain range of $m$. 

In \cite{bringmann2016dyson}, the authors introduce techniques in order to compute the bivariate asymptotic behaviour of coefficients for a Jacobi form in order to answer Dyson's conjecture on the bivariate asymptotic behaviour of the partition crank. This method is used in numerous other papers - for example, in relation to the rank of a partition \cite{dousse2014asymptotic}, ranks and cranks of cubic partitions \cite{kim2016asymptotic}, and certain genera of Hilbert schemes \cite{manschot2014asymptotic} (a result that has recently been extended to a complete classification with exact formulae using the Hardy-Ramanujan circle method \cite{gillman2019partitions}), along with many other partition-related statistics.

 Using Wright's circle method \cite{wright1934asymptotic,wright1971stacks} and following the same approach as \cite{bringmann2016dyson} we show the following theorem.

\begin{theorem}\label{Theorem: main}
	For $\beta \coloneqq \pi \sqrt{\frac{2}{n}}$ and $|m| \leq \frac{1}{6 \beta} \log(n)$ we have that
	\begin{equation*}
	b(m,n) =(-1)^{m+\delta+\frac{3}{2}} \frac{\beta^6 m}{8 \pi^5 (2n)^{\frac{1}{4} }} e^{2 \pi \sqrt{2n}} + O \left( m n^{-\frac{15}{4}} e^{2 \pi \sqrt{2n}} \right)
	\end{equation*}
	as $n \rightarrow \infty$. Here, $\delta \coloneqq 1$ if $m <0$ and $\delta = 0$ otherwise.
\end{theorem}

\begin{remark}
	Although our approach is similar to \cite{bringmann2016dyson,dousse2014asymptotic}, in some places we require a little more care since finding the Fourier coefficients requires taking an integral over a path where $f$ has a pole. In this case, we turn to the framework of \cite{dabholkar2012quantum} - this is explained explicitly in Section \ref{Section: asymptotics}.
\end{remark}

We begin in Section \ref{Section: prelims} by recalling relevant results that are pertinent to the rest of the paper. In Section \ref{Section: bounds toward dominant pole} we investigate the behaviour of $f$ toward the dominant pole $q = 1$. We follow this in Section \ref{Section: bounds away from dominant pole} by bounding the contribution away from the pole at $q=1$. We finish in Section \ref{Section: Circle method} by applying Wright's circle method to find the asymptotic behaviour of $b(m,n)$ and hence prove Theorem \ref{Theorem: main}.

\section*{acknowledgements}
The author would like to thank Kathrin Bringmann for initially suggesting the project, as well as insightful conversations and useful comments on the contents of the paper. The author would also like to thank the referee for numerous helpful comments and suggestions.

\section{Preliminaries}\label{Section: prelims}
Here we recall relevant definitions and results which will be used throughout the rest of the paper. 

\subsection{Properties of $\vartheta$ and $\eta$}
When determining the asymptotic behaviour of $f$ we will require the modularity behaviour of both $\vartheta$ and $\eta$. It is well-known that $\vartheta$ satisfies the following lemma (see e.g. \cite{mumford2007tata}).
\begin{lemma}\label{Lemma: transformation of theta}
	The function $\vartheta$ satisfies the following transformation properties.\\
	\begin{enumerate}
		\item	 $\vartheta(-z ; \tau) = -\vartheta(z;\tau)$\\
		
		\item	$\vartheta(z+1;\tau) = -\vartheta(z;\tau)$\\
		
		\item $		\vartheta(z; \tau) = \frac{i}{\sqrt{-i \tau}} e^{ \frac{- \pi i z^2}{\tau}} \vartheta\left( \frac{z}{\tau} ; -\frac{1}{\tau} \right)$
	\end{enumerate}
\end{lemma}

Further, we have the following well-known modular transformation formula of $\eta$ (see e.g. \cite{siegel_1954}).

\begin{lemma}\label{Lemma: transformation of eta}
	We have that
	\begin{equation*}
	\begin{split}
	\eta\left( -\frac{1}{\tau} \right) = \sqrt{-i\tau} \eta(\tau).
	\end{split}
	\end{equation*}
\end{lemma}

\subsection{Euler Polynomials}
We will also make use properties of the Euler polynomials $E_r$, defined by the generating function
\begin{equation*}
\frac{2e^{xt}}{1+e^t} \eqqcolon \sum_{r \geq 0} E_r(x) \frac{t^r}{r!}.
\end{equation*}
Lemma 2.2 of \cite{bringmann2016dyson} shows that the following lemma holds.
\begin{lemma}\label{Lemma: sech in terms of E}
	We have
	\begin{equation*}
	-\frac{1}{2} \sech^2 \left( \frac{t}{2} \right) = \sum_{r \geq 0} E_{2r+1} (0) \frac{t^{2r}}{(2r)!}.
	\end{equation*}
\end{lemma}
Further, Lemma 2.3 of \cite{bringmann2016dyson} gives the following integral representation for the Euler polynomials.
\begin{lemma}
	We have that
	\begin{equation*}
	\mathcal{E}_j \coloneqq \int_{0}^{\infty} \frac{w^{2j+1}}{\sinh(\pi w)} dw = \frac{(-1)^{j+1} E_{2j+1} (0)}{2}.
	\end{equation*}
\end{lemma}

\subsection{A particular bound}
In Section \ref{Section: bounds away from dominant pole} we require a bound on the size of 
\begin{equation*}
P(q) \coloneqq \frac{q^{\frac{1}{24}}}{\eta(\tau)},
\end{equation*}
away from the pole at $q=1$. For this we use the following lemma which is shown to hold in Lemma 3.5 of \cite{bringmann2016dyson}.
\begin{lemma}
	Let $\tau = u +iv \in \H$ with $Mv \leq u \leq \frac{1}{2}$ for $u>0$ and $v \rightarrow 0$. Then
	\begin{equation*}
	|P(q)| \ll \sqrt{v} \exp \left[ \frac{1}{v} \left(\frac{\pi}{12} - \frac{1}{2\pi} \left(1- \frac{1}{\sqrt{1+M^2}}\right)\right) \right].
	\end{equation*}
\end{lemma}
In particular, with $v = \frac{\beta}{2\pi}$, $u=\frac{\beta m^{-\frac{1}{3}} x}{2 \pi}$ and $M = m^{-\frac{1}{3}}$ this gives for $1 \leq x \leq \frac{\pi m^{\frac{1}{3}}}{\beta}$ the bound
\begin{equation}\label{Equation: bound on P(q)}
|P(q)| \ll n^{-\frac{1}{4}} \exp \left[ \frac{2 \pi}{\beta} \left( \frac{\pi}{12} - \frac{1}{2\pi} \left( 1- \frac{1}{\sqrt{1 + m^{-\frac{2}{3}}}} \right) \right) \right].
\end{equation}

\subsection{$I$-Bessel functions}
Here we recall relevant results on the $I$-Bessel function defined by
\begin{equation*}
I_{\ell} (x) \coloneqq \frac{1}{2 \pi i} \int_{\Gamma} t^{-\ell - 1}e^{\frac{x}{2} \left(t+\frac{1}{t} \right)} dt,
\end{equation*}
where $\Gamma$ is a contour which starts in the lower half plane at $-\infty$, surrounds the origin counterclockwise and returns to $-\infty$ in the upper half-plane. We are particularly interested in the asymptotic behaviour of $I_\ell$, given in the following lemma (see e.g. (4.12.7) of \cite{andrews_askey_roy_1999}).
\begin{lemma}\label{Lemma: asymptotic of I Bessel}
	For fixed $\ell$ we have
	\begin{equation*}
	I_\ell(x) = \frac{e^x}{\sqrt{2 \pi x}} + O\left(\frac{e^x}{x^{\frac{3}{2}}}\right)
	\end{equation*}
	as $x \rightarrow \infty$.
\end{lemma}

We also require the behaviour of an integral related to the $I$-Bessel function. Define
\begin{equation*}
P_{s} \coloneqq \frac{1}{2 \pi i} \int_{1-im^{-\frac{1}{3}}}^{1+im^{-\frac{1}{3}}} v^s e^{\pi \sqrt{2n} \left(v + \frac{1}{v}\right)} dv.
\end{equation*}
Then Lemma 4.2 of \cite{bringmann2016dyson} reads as follows.
\begin{lemma}\label{Lemma: P in terms of I Bessel}
	For $|m| \leq \frac{1}{6 \beta} \log(n)$ we have
	\begin{equation*}
	P_s = I_{-s-1}\left( 2 \pi \sqrt{2n}\right) + O\left( \exp\left(\pi \sqrt{2n}\left(1+ \frac{1}{1+ |m|^{-\frac{2}{3}}}\right)\right) \right)
	\end{equation*}
	as $n \rightarrow \infty$.
\end{lemma}

\section{Asymptotic behaviour of $f$}\label{Section: asymptotics}
The aim of this Section is to determine the asymptotic behaviour of $f$. To do so we consider two separate cases: when $q$ tends toward the pole $q=1$, and when $q$ is away from this pole. It is shown that the behaviour toward the pole at $q=1$ gives the dominant contribution when applying the circle method in Section \ref{Section: Circle method}.

First note that Lemma \ref{Lemma: transformation of theta} implies that $f(-z;\tau) = -f(z;\tau)$, and so $b(-m,n) = -b(m,n)$. We now restrict our attention to the case of $m\geq 0$.

We next find the Fourier coefficient of $\zeta^m$ of $f$, following the framework of \cite{dabholkar2012quantum}. Since there is a pole of $f$ at $z = \frac{1}{2}$, we define
\begin{equation*}
\begin{split}
f_m^\pm (\tau) \coloneqq & \int_{0}^{\frac{1}{2} - a} f(z;\tau) e^{-2 \pi i m z} dz + \int_{\frac{1}{2} + a}^{1} f(z;\tau) e^{-2 \pi i m z} dz + G^\pm \\
= & -2i \int_{0}^{\frac{1}{2} - a} f(z;\tau) \sin(2 \pi m z) dz + G^\pm,
\end{split}
\end{equation*}
where $a>0$ is small, and
\begin{equation*}
G^\pm \coloneqq \int_{\frac{1}{2}-a}^{\frac{1}{2}+a} f(z;\tau) e^{-2 \pi i m z} dz.
\end{equation*} 
For $G^+$ the integral is taken over a semi-circular path passing above the pole. Similarly, $G^-$ is taken over a semi-circular path passing below the pole. Then the Fourier coefficient of $\zeta^m$ of $f$ is
\begin{equation}\label{Equation: f_m split into two integrals}
f_m (\tau) \coloneqq \frac{f_m^+ + f_m^-}{2} = -2 i \lim_{a \rightarrow 0^+} \int_{0}^{\frac{1}{2} - a} f(z;\tau) \sin(2 \pi m z) dz  + \frac{G^+ +G^-}{2}.
\end{equation}

Letting $z \mapsto 1-z$ maps the path of integration from $0$ to $1$ passing above the pole to the path from $1$ to $0$ passing below the pole. One then sees directly that
\begin{align}\label{f_m(tau)}
f_m (\tau) = -2 i \int_{0}^{\frac{1}{2}} f(z;\tau) \sin(2 \pi m z) dz.
\end{align}

In the following two subsections we determine the asymptotic behaviour of $f$ toward and away from the dominant pole at $q=1$ respectively. Throughout, we will let $\tau = \frac{i \varepsilon}{2 \pi}$, $\varepsilon \coloneqq \beta(1 + ix m^{-\frac{1}{3}})$ and $\beta \coloneqq \pi \sqrt{\frac{2}{n}}$. We determine asymptotics as $n \rightarrow \infty$.
\subsection{Bounds towards the dominant pole}\label{Section: bounds toward dominant pole}
Here we find the asymptotic behaviour of $f$ toward the dominant pole at $q=1$, shown in the following lemma.

\begin{lemma}\label{Lemma: main asmyptotic term of f at 0}
	Let $\tau = \frac{i \varepsilon}{2\pi}$, with $0 < \re(\varepsilon) \ll 1$, and $0 < z < \frac{1}{2}$. Then we have that
	\begin{equation*}
	f\left(z;\frac{i \varepsilon}{2 \pi}  \right) = -\frac{\varepsilon^3}{\pi^3}  \frac{\sinh\left( \frac{2 \pi^2 z}{\varepsilon} \right)^4}{ \sinh\left( \frac{4 \pi^2 z}{\varepsilon} \right) }  \left( 1 +  e^{ -4\pi^2 \re\left(\frac{1}{\varepsilon}\right) (1-2z)} + O\left( e^{ -4\pi^2 \re\left(\frac{1}{\varepsilon}\right) (1-z)  }  \right)  \right).
	\end{equation*}
	
\end{lemma}
\begin{proof}
	Using the modularity of $f$ (which follows from Lemmas \ref{Lemma: transformation of theta} and \ref{Lemma: transformation of eta}) and setting $q_0 \coloneqq e^{-\frac{2 \pi i}{\tau}}$, we have that 
	\begin{equation*}
	\begin{split}
f(z;\tau) & = \frac{\tau^3 \zeta^{\frac{2}{\tau}} \prod_{n \geq 1} \left(1 - \zeta^{\frac{1}{\tau}} q_0^n\right)^4 \left( 1-\zeta^{-\frac{1}{\tau}} q_0^{n-1}  \right)^4      }{ i \zeta^{\frac{1}{\tau}}  \prod_{n \geq 1} \left( 1-q_0^n \right)^6 \left( 1-\zeta^{\frac{2}{\tau}} q_0^n  \right) \left( 1-\zeta^{\frac{-2}{\tau}} q_0^{n-1} \right)     } \\
& = \frac{\tau^3  \left( \zeta^{\frac{1}{2\tau}} - \zeta^{-\frac{1}{2\tau}} \right)^4 }{ i \left( \zeta^{\frac{1}{\tau}} - \zeta^{\frac{-1}{\tau}} \right)  } \prod_{n \geq 1} \frac{   \left(1 - \zeta^{\frac{1}{\tau}} q_0^n\right)^4 \left( 1-\zeta^{-\frac{1}{\tau}} q_0^n  \right)^4      }{ \left( 1-q_0^n \right)^6 \left( 1-\zeta^{\frac{2}{\tau}} q_0^n  \right) \left( 1-\zeta^{\frac{-2}{\tau}} q_0^n \right)     }.
	\end{split}
	\end{equation*}
	This gives
	\begin{equation*}
	-\frac{\varepsilon^3}{ \pi^3} \frac{\sinh\left( \frac{2 \pi^2 z}{\varepsilon} \right)^4}{ \sinh\left( \frac{4 \pi^2 z}{\varepsilon} \right) } \prod_{n \geq 1} \frac{   \left(1 - e^{\frac{4\pi^2}{\varepsilon}(z-n) }\right)^4 \left( 1-e^{\frac{4\pi^2}{\varepsilon}(-z-n) }  \right)^4      }{ \left( 1-e^{\frac{-4\pi^2 n}{\varepsilon} } \right)^6 \left( 1-e^{\frac{4\pi^2}{\varepsilon}(2z-n) }  \right) \left( 1-e^{\frac{4\pi^2}{\varepsilon}(-2z-n) } \right)     }.
	\end{equation*}
	
	In order to find a bound we expand the denominator using geometric series. For $0 < z < \frac{1}{2}$ we see that $| e^{\frac{4\pi^2}{\varepsilon} (\pm 2z - n)} |<1$ for all $n \geq 1$, and so we expand the denominator to obtain the product as
	\begin{equation*}
	\prod_{n \geq 1}   \left(1- e^{\frac{4\pi^2 }{\varepsilon} (z-n)}   \right)^4 \left( 1 - e^{\frac{-4\pi^2 }{\varepsilon} (z+n) }  \right)^4   \sum_{j \geq 0} e^{\frac{4 j \pi^2}{\varepsilon} (2z - n)} \sum_{k \geq 0} e^{\frac{- 4 k \pi^2}{\varepsilon} (2z + n)} \left(\sum_{\ell \geq 0} e^{\frac{-4 \pi^2 \ell n}{\varepsilon}} \right)^6,
	\end{equation*}
	which, for $0 < \re(\varepsilon) \ll 1$, is of order 
	\begin{equation*}
	1 + e^{ -4\pi^2 \re\left(\frac{1}{\varepsilon}\right) (1-2z)} + O\left( e^{ -4\pi^2 \re\left(\frac{1}{\varepsilon}\right) (1-z)  }  \right).
	\end{equation*}
	 Hence overall we find that 
	\begin{equation*}
	f\left(z;\frac{i \varepsilon}{2 \pi}  \right) =  -\frac{\varepsilon^3}{\pi^3}  \frac{\sinh\left( \frac{2 \pi^2 z}{\varepsilon} \right)^4}{ \sinh\left( \frac{4 \pi^2 z}{\varepsilon} \right) }  \left( 1 +  e^{ -4\pi^2 \re\left(\frac{1}{\varepsilon}\right) (1-2z)} + O\left( e^{ -4\pi^2 \re\left(\frac{1}{\varepsilon}\right) (1-z)  }  \right) \right),
	\end{equation*} 
	yielding the claim.
\end{proof}
\begin{remark}
	It is easy to see that this gives the same main term as noted in Section 4.5 of \cite{he2015notes} (up to sign, which the authors there do not make use of).
\end{remark}
Since $f(z;\tau) = - f(1-z;\tau)$ we see this immediately also implies the following lemma.
\begin{lemma}
	Let $\tau = \frac{i \varepsilon}{2\pi}$, with $0 < \re(\varepsilon) \ll 1$, and $\frac{1}{2} < z < 1$. Then we have that
	\begin{equation*}
	f\left(z;\frac{i \varepsilon}{2 \pi}  \right) = \frac{\varepsilon^3}{\pi^3}  \frac{\sinh\left( \frac{2 \pi^2 (1-z)}{\varepsilon} \right)^4}{ \sinh\left( \frac{4 \pi^2 (1-z)}{\varepsilon} \right) }  \left( 1 + e^{ -4\pi^2 \re\left(\frac{1}{\varepsilon}\right) (2z-1)} + O\left( e^{ -4\pi^2 \re\left(\frac{1}{\varepsilon}\right) z  }  \right)  \right).
	\end{equation*}
\end{lemma}

We then obtain the following theorem.
\begin{theorem}\label{Theorem: behaviour of f_m at dominant pole}
	For $|x| \leq 1$ we have that
	\begin{equation*}
	f_m\left( \frac{i \varepsilon}{2 \pi} \right) = (-1)^{m+\frac{3}{2}} \frac{m}{4 \pi^4} \varepsilon^5 e^{\frac{2\pi^2}{\varepsilon}} +O(\varepsilon^3)
	\end{equation*}
	as $n \rightarrow \infty$.
\end{theorem}
\begin{proof}
	We have 
	\begin{align*}
	-\frac{\varepsilon^3}{\pi^3} \frac{\sinh\left(\frac{2\pi^2z}{\varepsilon}\right)^4}{\sinh\left(\frac{4\pi^2}{\varepsilon}\right)} =- \frac{\varepsilon^3}{\pi^3} e^{\frac{4\pi^2z}{\varepsilon}} \left(1+O(e^{-\frac{4 \pi^2z}{\varepsilon}}) \right).
	\end{align*}
	Plugging this into \eqref{f_m(tau)} and integrating explicitly, using the formula
	\begin{align*}
	\int_0^{\frac{1}{2}} e^{xz} \sin(2\pi m z) dz = - \frac{2 \pi m}{4\pi^2 m^2 + x^2} e^{\frac{x}{2}} \cos(\pi m) - \frac{2 \pi m}{4\pi^2m^2 + x^2},
	\end{align*}
	shows directly that
	\begin{align*}
	f_m(\tau) = & -2\frac{\varepsilon^3}{\pi^3}(-1)^{m+\frac{1}{2}}  \frac{2 \pi m}{\left(\frac{16 \pi^2}{\varepsilon^2}\right)} e^{\frac{2 \pi^2}{\varepsilon}} + O(\varepsilon^3)= (-1)^{m+\frac{3}{2}} \frac{m}{4 \pi^4} \varepsilon^5 e^{\frac{2\pi^2}{\varepsilon}} +O(\varepsilon^3).
	\end{align*}
\end{proof}

\subsection{Bounds away from the dominant pole}\label{Section: bounds away from dominant pole}
 We next investigate the behaviour of $f_m$ away from the pole $q=1$, by assuming that $1 \leq x \leq \frac{\pi m^{\frac{1}{3}}}{\beta}$. In the following lemma we bound the term
 \begin{equation*}
 \frac{\eta(2\tau)^8}{\eta(\tau)^{16}},
 \end{equation*}
 away from the pole $q = 1$.
 \begin{lemma}
 	For $1 \leq x \leq \frac{\pi m^{\frac{1}{3}}}{\beta}$ we have that
 	\begin{equation*}
 	\left| \frac{\eta\left( \frac{i \varepsilon}{\pi}\right)^8}{\eta\left( \frac{i \varepsilon}{2 \pi} \right)^{16}} \right| \ll n^{-2} \exp\left[    \pi \sqrt{2n} - \frac{8 \sqrt{2n}}{\pi} \left(1 - \frac{1}{\sqrt{1+ m^{-\frac{2}{3}}}}\right) \right]
 	\end{equation*}
 	as $n \rightarrow \infty$.
 \end{lemma}
\begin{proof}
	We first write 
	\begin{equation*}
	\frac{\eta(2\tau)^8}{\eta(\tau)^{16}} = \frac{\eta(2 \tau)^8}{q^{\frac{2}{3}}} \frac{q^{\frac{2}{3}}}{\eta(\tau)^{16}}.
	\end{equation*} 
	Using equation \eqref{Equation: bound on P(q)} directly we find that (with $\tau = \frac{i \varepsilon}{2 \pi}$)
	\begin{equation*}
	\frac{q^{\frac{2}{3}}}{\eta(\tau)^{16}} = P(q)^{16} \ll n^{-4} \exp\left[  \frac{  4 \pi \sqrt{2n}}{3} - \frac{8 \sqrt{2n}}{\pi} \left(1 - \frac{1}{\sqrt{1+ m^{-\frac{2}{3}}}}\right) \right]. 
	\end{equation*}
	It remains to consider the behaviour of $e^{\varepsilon/12} \eta(\frac{i\varepsilon}{ \pi})$. Using the transformation formula of $\eta$ given in Lemma \ref{Lemma: transformation of eta} along with the well-known summation representation of $\eta$, we see that as $n \rightarrow \infty$ we obtain
	\begin{align*}
	e^{\frac{\varepsilon}{12}} \eta\left(\frac{i\varepsilon}{ \pi} \right) = e^{\frac{\varepsilon}{12}} \sqrt{\frac{ \pi}{\varepsilon}} e^{-\frac{\pi^2}{12 \varepsilon}} \sum_{j \in \Z} (-1)^j e^{- \frac{ \pi^2(3j^2 - j)}{\varepsilon} } \ll \sqrt{\frac{ \pi}{\varepsilon}} e^{-\frac{\pi^2}{12 \varepsilon}}.
	\end{align*}
	We hence have
	\begin{align*}
	\left| \frac{\eta\left(\frac{i \varepsilon}{ \pi}\right)}{e^{\frac{2\varepsilon}{3}}} \right|^8 \ll \left| \sqrt{\frac{ \pi}{\varepsilon}} e^{-\frac{\pi^2}{12 \varepsilon}} \right|^8 \ll \left(\frac{ \pi}{\beta}\right)^4 e^{-\frac{2\pi^2}{3 \beta}} \ll n^{2} e^{-\frac{\pi \sqrt{2n}}{3}}.
	\end{align*}
	Combining the two bounds yields the result.
\end{proof}

 Next, we investigate the contribution of
 \begin{equation*}
 \left| \int_{0}^{\frac{1}{2} - a} f(z;\tau) \sin(2\pi m z) dz \right| \ll \int_{0}^{\frac{1}{2} - a} |f(z;\tau) \sin(2 \pi m z)| dz.
 \end{equation*}
 Then we want to bound
 \begin{equation*}
 \begin{split}
 |f(z;\tau) \sin(2\pi m z)| = \left| \frac{ \sin(2 \pi m z) \vartheta(z;\tau)^4}{\eta(\tau)^9 \vartheta(2z;\tau)} \right|
 \end{split}
 \end{equation*}
 away from the dominant pole. For $0 < b < \frac{1}{2}$ far from $\frac{1}{2}$ we see that we may bound the integrand in modulus by
 \begin{equation*}
 |f(b;\tau) \sin(2 \pi m b)| \ll |P(q)|^9 \left| q^{-\frac{3}{8}} \frac{\vartheta(b;\tau)^4}{\vartheta(2b;\tau)} \right| \ll |P(q)|^9 \sum_{n \in \Z} |q|^{\frac{n^2 +n}{2}} \ll |P(q)|^9 \sum_{n \in \Z} e^{-\beta n^2}. 
 \end{equation*}
 
  As $z \rightarrow \frac{1}{2}$ we apply L'H\^{o}pital's rule to the integrand $|f(z;\tau) \sin(2 \pi m z)|$ which yields the bound 
 \begin{equation*}
   \left| \int_{0}^{\frac{1}{2} - a} f(z;\tau) \sin(2\pi m z) dz \right| \ll \frac{\eta(2\tau)^8}{\eta(\tau)^{16}}.
 \end{equation*}
 
 Hence, away from the dominant pole in $q$, we have shown the following proposition.
 
 \begin{proposition}\label{Proposition: f away from dominant}
 	For $1 \leq x \leq  \frac{\pi m^{\frac{1}{3}}}{\beta}$ we have that
 \begin{equation*}
 \left|f\left(z; \frac{i \varepsilon}{2 \pi} \right) \right| \ll n^{-2} \exp\left[  \pi \sqrt{2n} - \frac{8 \sqrt{2n}}{\pi} \left(1 - \frac{1}{\sqrt{1+ m^{-\frac{2}{3}}}}\right) \right]
 \end{equation*}
 as $n \rightarrow \infty$.
 \end{proposition}

 \section{The Circle Method}\label{Section: Circle method}
 In this section we use Wright's variant of the Circle Method to complete the proof of Theorem \ref{Theorem: main}. We start by noting that Cauchy's theorem implies that
 \begin{equation*}
 b(m,n) = \frac{1}{2 \pi i} \int_{C} \frac{f_m (\tau)}{q^{n+1}} dq,
 \end{equation*}
 where $C \coloneqq \{ q \in \C \mid |q| = e^{-\beta}  \}$ is a circle centred at the origin of radius less than $1$, with the path taken in the counter-clockwise direction. 
 Making a change of variables, changing the direction of the path of the integral, and recalling that $\varepsilon = \beta(1+ixm^{-\frac{1}{3}})$ we have
 \begin{equation*}
 b(m,n) = \frac{\beta}{2\pi m^{\frac{1}{3}}} \int_{|x| \leq \frac{\pi m^{\frac{1}{3}}}{\beta}} f_m\left(\frac{i \varepsilon}{2 \pi}\right) e^{\varepsilon n} dx.
 \end{equation*}
 Splitting this integral into two pieces, we have $b(m,n) = M + E$ where
 \begin{equation*}
 M \coloneqq \frac{\beta}{2 \pi m^{\frac{1}{3}}} \int_{|x| \leq 1} f_m\left( \frac{i \varepsilon}{2 \pi}\right) e^{\varepsilon n} dx,
 \end{equation*}
 and
 \begin{equation*}
 E \coloneqq \frac{\beta}{2 \pi m^{\frac{1}{3}}} \int_{1 \leq |x| \leq \frac{\pi m^{\frac{1}{3}}}{\beta}} f_m\left(\frac{i \varepsilon}{2 \pi}\right) e^{\varepsilon n} dx.
 \end{equation*}
 
  Next we determine the contributions of each of the integrals $M$ and $E$, and see that $M$ contributes to the main asymptotic term, while $E$ is part of the error term.
 
 \subsection{The major arc}
 First we concentrate on the contribution $M$. Then we obtain the following proposition.
 
 \begin{proposition}
 	We have that
 	\begin{equation*}
 	M = (-1)^{m+\frac{3}{2}} \frac{\beta^6 m}{8 \pi^5 (2n)^{\frac{1}{4} }} e^{2 \pi \sqrt{2n}} + O \left( m n^{-\frac{13}{4}} e^{2 \pi \sqrt{2n}} \right)
 	\end{equation*}
 	as $n \rightarrow \infty$.
 \end{proposition}
 \begin{proof}
 	By Theorem \ref{Theorem: behaviour of f_m at dominant pole} and making the change of variables $v = 1 +ixm^{-\frac{1}{3}}$ we obtain  
 	\begin{equation*}
 	M = (-1)^{m+\frac{3}{2}} \frac{\beta^6 m}{4 \pi^4 } P_{4} + O\left( \beta^3 e^{\pi \sqrt{2n}}  \right).
 	\end{equation*}
 	Now we rewrite $P_{4}$ in terms of the $I$-Bessel function using Lemma \ref{Lemma: P in terms of I Bessel}, yielding
 	\begin{equation*}
 	\begin{split}
 	M =(-1)^{m+\frac{3}{2}} \frac{\beta^6 m}{4 \pi^4 } I_{-5}(2 \pi \sqrt{2n}) + O\left( \beta^6 e^{\pi \sqrt{2n} \left(1+\frac{1}{1+m^{-\frac{2}{3}}}\right)} \right) + O\left( \beta^3 e^{\pi \sqrt{2n}}  \right).
 	\end{split}
 	\end{equation*}
 	The asymptotic behaviour of the $I$-Bessel function given in Lemma \ref{Lemma: asymptotic of I Bessel} gives that
 	\begin{equation*}
 	\begin{split}
 	M = &  (-1)^{m+\frac{3}{2}} \frac{\beta^6 m}{8 \pi^5 (2n)^{\frac{1}{4} }} e^{2 \pi \sqrt{2n}} + O \left( m n^{-\frac{15}{4}} e^{2 \pi \sqrt{2n}} \right) +  O\left( \beta^6 e^{\pi \sqrt{2n} \left(1+\frac{1}{1+m^{-\frac{2}{3}}}\right)} \right) \\
 	& +  O\left( \beta^3 e^{\pi \sqrt{2n}}  \right).
 	\end{split}
 		\end{equation*}
 		It is clear that the first error term is the dominant one, and the result follows.
 \end{proof}

 \subsection{The error arc}
 Now we bound $E$ as follows.
 \begin{proposition}
 As $n \rightarrow \infty$
 	\begin{equation*}
 	E \ll n^{-2} \exp\left[ 2\pi \sqrt{2n} - \frac{8 \sqrt{2n}}{\pi} \left(1 - \frac{1}{\sqrt{1+ m^{-\frac{2}{3}}}}\right) \right].
 	\end{equation*}
 \end{proposition}
 \begin{proof}
 	By Proposition \ref{Proposition: f away from dominant} we see that the main term in the error arc is given by the residue. Hence we may bound
 	\begin{equation*}
 	\begin{split}
 	E & \ll \int_{1 \leq x \leq  \frac{\pi m^{\frac{1}{3}}}{\beta} } n^{-2} \exp\left[    \pi \sqrt{2n} - \frac{8 \sqrt{2n}}{\pi} \left(1 - \frac{1}{\sqrt{1+ m^{-\frac{2}{3}}}}\right) \right] e^{\varepsilon n} dx\\
 	& \ll n^{-2} \exp\left[ 2\pi \sqrt{2n} - \frac{8 \sqrt{2n}}{\pi} \left(1 - \frac{1}{\sqrt{1+ m^{-\frac{2}{3}}}}\right) \right].
 	\end{split}
 	\end{equation*}
 	Noting that this is exponentially smaller than $M$ finishes the proof of Theorem \ref{Theorem: main}.
 \end{proof}

\section{Open questions}
We end by commenting on some questions related to the results presented above.
\begin{enumerate}
	\item Here we discuss the asymptotic profile of the coefficients $b(m,n)$ for $|m| \leq \frac{1}{6 \beta} \log(n)$. We are also interested in the profile when $m$ is larger than this bound, and so in future it would be instructive to investigate the asymptotic profile of $b(m,n)$ for large $|m|$. For example, similar results in this direction for the crank of a partition are given in \cite{parry2017dyson}.
	
	\item In the present paper, we provide a framework for investigating the profile of eta-theta quotients. In particular, we deal with the case of a function with a single simple pole on the path of integration. Future research is planned in order to expand this framework for a family of meromorphic eta-theta quotients with a finite number of (not necessarily single) poles on the path of integration. This should include similar eta-theta quotients that appear in other physical partition functions.
	
	\item In showing Theorem \ref{Theorem: main} we see that the main asymptotic term arises from the pole at $z = 1/2$, and in turn from the residue term $\frac{\eta(2\tau)^8}{\eta(\tau)^{16}}$; is there a physical interpretation for the fact that these terms give the largest contribution to the asymptotic behaviour of $b(m,n)$?
\end{enumerate}

\end{document}